\documentclass[11pt]{article}
\usepackage[margin=1in]{geometry}
\usepackage{amsmath,amssymb,amsthm,mathtools}
\usepackage[hidelinks]{hyperref}
\usepackage{booktabs}
\usepackage{tikz}
\usepackage{enumitem}
\usetikzlibrary{positioning,shapes.geometric,calc}

\definecolor{domcolor}{RGB}{220,50,50}

\tikzset{
  vertex/.style={circle,draw=black,fill=white,minimum size=6pt,inner sep=0pt,line width=0.6pt},
  dom/.style={circle,draw=black,fill=domcolor,minimum size=6pt,inner sep=0pt,line width=0.8pt},
  leaf/.style={circle,draw=black,fill=white,minimum size=5pt,inner sep=0pt,line width=0.5pt},
  edge/.style={line width=0.6pt,black},
}
\newtheorem{theorem}{Theorem}

\newtheorem{corollary}[theorem]{Corollary}

\theoremstyle{definition}
\newtheorem{definition}[theorem]{Definition}
\theoremstyle{remark}
\newtheorem{remark}[theorem]{Remark}

\newcommand{\gam}{\gamma}
\newcommand{\zet}{\zeta}

\title{Four Dominion Growth Regimes in Trees: Forcing, Fibonacci Enumeration, Periodicity, and Stability}
\author{Julian~Allagan$^{1,*}$, Erin~Gray$^{1}$, Jennifer~Sawyer$^{1}$, and Gabrielle Morgan$^{1}$  \\
  \small $^{1}$ Elizabeth City State University, Elizabeth City, NC, USA \\
 \small Department of Mathematics, Computer Science and Engineering Technology\\
  \small $^{*}$\textit{E-mail:} \texttt{adallagan@ecsu.edu}\qquad
  \small $^{*}$Corresponding author
}
\date{}

\begin{document}
\maketitle

\begin{abstract}
We study the \emph{dominion} $\zeta(G)$, defined as the number of minimum dominating sets of a graph $G$,
and analyze how local forcing and boundary effects control the flexibility of optimal domination in trees.
For path-based pendant constructions, we identify a sharp forcing threshold: attaching a single pendant
vertex to each path vertex yields complete independence with $\zeta=2^{\gamma}$, whereas attaching two or
more pendant vertices forces a unique minimum dominating set.
Between these extremes, sparse pendant patterns produce intermediate behavior: removing endpoint pendants
gives $\zeta=2^{\gamma-2}$, while alternating pendant attachments induce Fibonacci growth
$\zeta\asymp \varphi^{\gamma}$, where $\varphi$ is the golden ratio.
For complete binary trees $T_h$, we establish a rigid period-$3$ law $\zeta(T_h)\in\{1,3\}$ despite
exponential growth in $|V(T_h)|$.
We further prove a sharp stability bound under leaf deletions,
$\zeta(T_h-X)\le 2^{m_1(X)}\zeta(T_h)$, where $m_1(X)$ counts parents that lose exactly one child;
in particular, deleting a single leaf preserves the domination number and exactly doubles the dominion.
\end{abstract}

\section{Introduction}

The domination number $\gam(G)$ measures the minimum number of vertices required to dominate a graph $G$,
where a set $D\subseteq V(G)$ is dominating if every vertex lies in $D$ or is adjacent to a vertex of $D$.
While $\gam(G)$ captures optimal cost, it does not distinguish between \emph{rigid} and \emph{flexible}
optima. This distinction is quantified by the \emph{dominion}
\[
\zet(G)=\#\{D\subseteq V(G): |D|=\gam(G)\ \text{and $D$ dominates $G$}\},
\]
which counts all minimum dominating sets of $G$. Introduced in this form by
Allagan and Bobga~\cite{AllaganBobga2021}, the dominion measures the degree of freedom among optimal
solutions and reveals how local structural constraints propagate globally.

Dominating sets arise naturally in network design and control, where vertices represent facilities,
sensors, or agents that must collectively monitor or influence an entire system~\cite{HaynesHedetniemiSlater1998}.
In this context, the dominion $\zet(G)$ quantifies solution diversity and thus robustness:
systems with $\zet(G)>1$ admit alternative optimal configurations under failures or reconfiguration.
Applications include wireless sensor networks~\cite{AnithaSebastian2011,Wu1999,Dai2004}, where connected
dominating sets form efficient routing backbones; Internet of Things architectures~\cite{Bendouda2018,Jia2009},
where domination-based clustering supports adaptive topology control; distributed computing~\cite{Kuhn2004},
where minimal dominating sets serve as leader sets; and network security~\cite{Angel2022,Kang2008}, where
monitor placement mitigates coordinated attacks. Beyond communication networks, facility location and
resource allocation models~\cite{Church1974,Farahani2010} use the multiplicity of optimal placements to
assess resilience under demand fluctuations or site failures. Recent extensions include influence
maximization in social networks~\cite{Li2013} and epidemic control~\cite{Gao2014}, where $\zet(G)$ reflects
the flexibility of intervention strategies.

The enumeration of minimum dominating sets has been studied along several largely independent lines.
Early work focused on exact counts for specific graph families, including paths, cycles, and related
constructions~\cite{AllaganBobga2021}, with later extensions to grid graphs~\cite{SuAllagan2024}.
A complementary extremal direction asks how large $\zet(G)$ can be as a function of $\gam(G)$.
Godbole, Jamieson, and Jamieson~\cite{GodboleJamiesonJamieson2014}, followed by Connolly
\emph{et al.}~\cite{ConnollyGaborGodboleKay2016}, established general exponential upper bounds on the
number of minimum dominating sets.

At the opposite extreme, Goddard and Henning~\cite{GoddardHenning2023} characterized graphs admitting a
\emph{unique} minimum dominating set, identifying structural conditions under which optimal domination
is completely forced. The most definitive recent result for forests is due to Petr, Portier, and
Versteegen~\cite{PetrPortierVersteegen2024}, who proved that every forest $F$ satisfies
$\zet(F)\le 5^{\gam(F)}$ and constructed families achieving growth on the order of
$5^{2\gam(F)/5}$, establishing the correct exponential rate up to constant factors.

The present paper operates in a complementary \emph{structural} regime. Rather than maximizing $\zet(G)$
over all trees with fixed $\gam(G)$, we determine exact values of $\gam$ and $\zet$ for canonical tree
families and isolate the local forcing mechanisms responsible for distinct growth behaviors.
Specifically, we identify four qualitatively different regimes:
\begin{enumerate}[label=(\roman*),leftmargin=*]
\item \emph{Exponential freedom}, $\zet=2^{\gam}$, arising from independent local choices;
\item \emph{Fibonacci growth}, $\zet\asymp \varphi^{\gam}$, induced by coupled constraints;
\item \emph{Periodic rigidity}, where $\zet\in\{1,3\}$ for complete binary trees despite exponential
vertex growth; and
\item \emph{Complete forcing}, where local density collapses the dominion to $\zet=1$.
\end{enumerate}
In addition, we quantify stability under boundary perturbations, showing that deleting a single leaf
from a complete binary tree preserves the domination number while exactly doubling the dominion, and
more generally proving bounds of the form
$\zet(T_h-X)\le 2^{m_1(X)}\zet(T_h)$, where $m_1(X)$ counts parents losing exactly one child.

\begin{remark}[Structural versus extremal regimes]
The extremal bounds of Petr, Portier, and Versteegen~\cite{PetrPortierVersteegen2024}, together with earlier
results~\cite{GodboleJamiesonJamieson2014,ConnollyGaborGodboleKay2016}, show that the dominion of forests
can grow exponentially with base at most $5$ in the domination number. Our results demonstrate that many
natural tree families lie far below this envelope due to strong local forcing and boundary effects.
In particular, complete binary trees satisfy $\zet(T_h)\in\{1,3\}$ for all $h$, even though $|V(T_h)|$
grows exponentially. The stability results proved here describe precisely how and when such rigid behavior
breaks under minimal perturbations, providing a structural counterpart to the global extremal theory.
\end{remark}

\begin{definition}
A set $D\subseteq V(G)$ is \emph{dominating} if every vertex lies in $D$ or has a neighbor in $D$.
The \emph{domination number} $\gam(G)$ is the minimum size of a dominating set.
The \emph{dominion} $\zet(G)$ is the number of dominating sets of size $\gam(G)$.
\end{definition}

Throughout, $P_n$ denotes the path on vertices $v_1,\dots,v_n$ with edges $v_iv_{i+1}$ $(1\le i\le n-1)$.
For $r\ge 1$ and $1\le i\le n$, let
\[
L_i=\{\ell_{i,1},\dots,\ell_{i,r}\}
\]
denote the set of $r$ pendant vertices attached to $v_i$ in the uniform attachment model.
We write $(F_t)_{t\ge 1}$ for the Fibonacci numbers, defined by $F_1=F_2=1$ and
$F_{t+1}=F_t+F_{t-1}$.

\section{Pendant path structures}
We first analyze path-based trees with pendant attachments. In these families, domination is governed by
local two-vertex (or $(r+1)$-vertex) clusters, so the transition between flexibility and forcing can be
read off directly from the attachment pattern.

\subsection{Uniform attachment: a forcing dichotomy}
\begin{theorem}[Forcing dichotomy for paths with pendant vertices]\label{thm:uniform}
Fix integers $n\ge 1$ and $r\ge 1$. Let $G(n,r)$ be the graph obtained from $P_n$ by attaching $r$
pendant vertices $\ell_{i,1},\dots,\ell_{i,r}$ to each $v_i$. Then
\[
\gamma(G(n,r))=n,
\qquad
\zeta(G(n,r))=
\begin{cases}
2^n, & r=1,\\[2pt]
1, & r\ge 2.
\end{cases}
\]
In particular, if $r\ge 2$ the unique minimum dominating set is $\{v_1,\dots,v_n\}$, whereas if $r=1$
every minimum dominating set is obtained by choosing one vertex from each pair
$\{v_i,\ell_{i,1}\}$ independently.
\end{theorem}

\begin{proof}
For each $i$, every pendant vertex $\ell\in L_i$ satisfies $N[\ell]=\{\ell,v_i\}$.
Consequently, any dominating set $D$ must meet the local cluster
\[
\{v_i\}\cup L_i \qquad (1\le i\le n).
\]
Since these $n$ clusters are pairwise disjoint, every dominating set has size at least $n$.
As the set $\{v_1,\dots,v_n\}$ dominates $G(n,r)$, it follows that $\gamma(G(n,r))=n$.

Let $D$ be a minimum dominating set, so $|D|=n$.
The disjointness of the clusters then forces
\[
|D\cap(\{v_i\}\cup L_i)|=1 \qquad (1\le i\le n).
\]

Assume first that $r\ge 2$.
If $v_i\notin D$, then none of the $r$ pendant vertices in $L_i$ is adjacent to any vertex outside
$\{v_i\}\cup L_i$, so each such pendant vertex must dominate itself.
This forces $L_i\subseteq D$, contradicting the condition above.
Hence $v_i\in D$ for all $i$, and the unique minimum dominating set is
$\{v_1,\dots,v_n\}$, giving $\zeta(G(n,r))=1$.

Now suppose $r=1$, and write $L_i=\{\ell_i\}$.
The condition above implies that $D$ chooses exactly one vertex from each pair
$\{v_i,\ell_i\}$.
Conversely, any such choice dominates $G(n,1)$, since the selected vertex dominates both $v_i$ and $\ell_i$,
and every path vertex is dominated either by itself or by its pendant.
There are $2^n$ independent choices, so $\zeta(G(n,1))=2^n=2^{\gamma(G(n,1))}$.
\end{proof}

This dichotomy is illustrated in Figure~\ref{fig:dominion-behaviors}(a,b) for $n=4$.
The next two corollaries are immediate specializations.

\begin{corollary}[Stars]\label{cor:stars}
For the star $K_{1,m}$ ($m\ge 1$),
\[
\gam(K_{1,m})=1,\qquad
\zet(K_{1,m})=
\begin{cases}
2,& m=1,\\
1,& m\ge 2.
\end{cases}
\]
\end{corollary}

\begin{proof}
If $m=1$ then $K_{1,1}\cong P_2$ and both vertices dominate. If $m\ge 2$, the center is the unique vertex
whose closed neighborhood equals $V(K_{1,m})$.
\end{proof}

\begin{corollary}[Full comb]\label{cor:comb}
Let $G_n$ denote the $n$-comb, the graph obtained from $P_n$ by attaching one pendant vertex to each $v_i$.
Then $\gam(G_n)=n$ and $\zet(G_n)=2^n=2^{\gam(G_n)}$.
\end{corollary}

\begin{proof}
This is Theorem~\ref{thm:uniform} with $r=1$.
\end{proof}
See Figure~\ref{fig:dominion-behaviors}(a) for three example minimum dominating sets of $G_4$.

\subsection{Interior pendants}
The uniform model isolates a density threshold. The next family shows that even at density $1$, boundary
effects can eliminate two degrees of freedom.

\begin{theorem}[Interior pendants]\label{thm:modified}
Let $G'_n$ be obtained from $P_n$ by attaching one pendant vertex $\ell_i$ to each internal vertex $v_i$
for $2\le i\le n-1$. Then
\[
\gam(G'_n)=\max\{1,n-2\},\qquad
\zet(G'_n)=
\begin{cases}
2,& n=2,\\
1,& n=3,\\
2^{\gam(G'_n)-2},& n\ge 4.
\end{cases}
\]
\end{theorem}

\begin{proof}
For $2\le i\le n-1$, set $C_i=\{v_i,\ell_i\}$. Since $N[\ell_i]=C_i$, every dominating set meets each $C_i$,
so $|D|\ge n-2$ for $n\ge 3$. As $\{v_2,\dots,v_{n-1}\}$ dominates $G'_n$, we obtain
$\gam(G'_n)=\max\{1,n-2\}$. The cases $n=2$ ($P_2$) and $n=3$ are immediate.

Assume $n\ge 4$ and let $D$ be a $\gam$-set, so $|D|=n-2$. Cluster forcing gives $|D\cap C_i|=1$ for all
$2\le i\le n-1$. If $\ell_2\in D$, then $v_1$ is not dominated by $\ell_2$ and must be dominated by an
additional vertex (necessarily $v_1$), contradicting $|D|=n-2$; hence $v_2\in D$. Similarly $v_{n-1}\in D$.
Therefore every $\gam$-set has the form
\[
D=\{v_2,v_{n-1}\}\cup \{x_i: 3\le i\le n-2\},\qquad x_i\in\{v_i,\ell_i\}.
\]
Each such choice dominates $G'_n$ because $v_2$ dominates $v_1,v_2,v_3$ and $\ell_2$, $v_{n-1}$ dominates
$v_{n-2},v_{n-1},v_n$ and $\ell_{n-1}$, and for $3\le i\le n-2$ the choice $x_i$ dominates $v_i$ and $\ell_i$.
Thus $\zet(G'_n)=2^{n-4}=2^{\gam(G'_n)-2}$.
\end{proof}


\subsection{Alternating combs and Fibonacci dominion}

We now consider a sparse attachment pattern in which local choices interact across adjacent clusters.
This coupling yields a Fibonacci recurrence and subexponential growth in $\gamma$.

\begin{definition}
For $n\ge 2$, let $E_n$ (resp.\ $O_n$) be the graph obtained from $P_n$ by attaching one pendant vertex
$\ell_i$ to each even (resp.\ odd) vertex $v_i$.
\end{definition}

\begin{theorem}[Fibonacci dominion for alternating combs]\label{thm:alt}
Let $n\ge 2$ and set $k=\lfloor n/2\rfloor$. Then
\[
\gamma(E_n)=k,\qquad
\zeta(E_n)=
\begin{cases}
F_{k+1},& n=2k,\\
F_k,& n=2k+1,
\end{cases}
\]
and
\[
\gamma(O_n)=\lceil n/2\rceil,\qquad
\zeta(O_n)=
\begin{cases}
F_{k+1},& n=2k,\\
F_{k+3},& n=2k+1,
\end{cases}
\]
where $(F_t)_{t\ge1}$ denotes the Fibonacci sequence defined by $F_1=F_2=1$ and
$F_{t+1}=F_t+F_{t-1}$.
\end{theorem}

\begin{proof}
We treat the even and odd attachment patterns separately.

\smallskip
\noindent\emph{Domination numbers.}
In $E_n$, each even index $2i$ contributes a pendant cluster
$C_{2i}=\{v_{2i},\ell_{2i}\}$ with $N[\ell_{2i}]=C_{2i}$.
Every dominating set must intersect each $C_{2i}$, and the clusters are pairwise disjoint.
Hence $\gamma(E_n)\ge k=\lfloor n/2\rfloor$.
Selecting all even hosts $\{v_2,v_4,\dots\}$ dominates $E_n$, so $\gamma(E_n)=k$.
The same argument, applied to odd indices, yields $\gamma(O_n)=\lceil n/2\rceil$.

\smallskip
\noindent\emph{Even attachment.}
For $t\ge 1$, let $H_t$ be the induced subgraph of $E_{2t}$ on path vertices
$v_1,\dots,v_{2t}$ together with pendants $\ell_2,\ell_4,\dots,\ell_{2t}$,
and write $a_t=\zeta(H_t)$.
Since $H_t$ contains exactly $t$ disjoint clusters $C_{2i}=\{v_{2i},\ell_{2i}\}$ and
$\gamma(H_t)=t$, every minimum dominating set intersects each cluster in exactly one vertex.

The initial values follow by inspection.
For $t=1$, $H_1$ is the path $v_1v_2$ with a pendant at $v_2$; the unique minimum
dominating set is $\{v_2\}$, so $a_1=1$.
For $t=2$, $H_2$ is the path $v_1v_2v_3v_4$ with pendants at $v_2$ and $v_4$.
Exactly two minimum dominating sets exist, namely $\{v_2,v_4\}$ and $\{v_2,\ell_4\}$,
so $a_2=2$.

Now fix $t\ge 3$ and let $D$ be a minimum dominating set of $H_t$.
Consider the final cluster $C_{2t}=\{v_{2t},\ell_{2t}\}$.

\emph{Case~1: $v_{2t}\in D$.}
Removing the vertices $\{v_{2t-1},v_{2t},\ell_{2t}\}$ and restricting $D$ to the
remaining induced subgraph yields a minimum dominating set of $H_{t-1}$.
Conversely, any minimum dominating set of $H_{t-1}$ extends uniquely to one of $H_t$
by adding $v_{2t}$.
Thus this case contributes $a_{t-1}$ choices.

\emph{Case~2: $\ell_{2t}\in D$.}
Then $v_{2t}\notin D$, and $v_{2t-1}$ must be dominated by a neighbor in $D$.
Since $v_{2t-1}$ is adjacent only to $v_{2t-2}$ and $v_{2t}$, it follows that
$v_{2t-2}\in D$ is forced.
Removing the last two even clusters and restricting to the induced subgraph on
$v_1,\dots,v_{2t-4}$ yields a minimum dominating set of $H_{t-2}$, and this correspondence
is bijective.
Hence this case contributes $a_{t-2}$ choices.

Combining the two cases gives the recurrence $a_t=a_{t-1}+a_{t-2}$ for $t\ge 3$,
with $a_1=1$ and $a_2=2$, and therefore $a_t=F_{t+1}$.
If $n=2k$, then $E_n=H_k$ and $\zeta(E_{2k})=F_{k+1}$.
If $n=2k+1$, the endpoint $v_{2k+1}$ has no pendant and forces $v_{2k}$ in every
minimum dominating set, reducing to $H_{k-1}$ and yielding $\zeta(E_{2k+1})=F_k$.

\smallskip
\noindent\emph{Odd attachment.}
For $t\ge 0$, let $J_t$ be the induced subgraph of $O_{2t+1}$ on vertices
$v_1,\dots,v_{2t+1}$ together with pendants $\ell_1,\ell_3,\dots,\ell_{2t+1}$,
and write $b_t=\zeta(J_t)$.
The graph $J_t$ contains $t+1$ disjoint clusters
$\{v_{2i-1},\ell_{2i-1}\}$, so $\gamma(J_t)=t+1$.

For $t=0$, $J_0$ consists of $v_1$ with a pendant $\ell_1$, and the unique minimum
dominating set is $\{v_1\}$; thus $b_0=1$.
For $t=1$, $J_1$ is the path $v_1v_2v_3$ with pendants at $v_1$ and $v_3$.
Exactly three minimum dominating sets exist, namely
$\{v_1,v_3\}$, $\{v_1,\ell_3\}$, and $\{\ell_1,v_3\}$, so $b_1=3$.

Fix $t\ge 2$ and let $D$ be a minimum dominating set of $J_t$.
Consider the final cluster $\{v_{2t+1},\ell_{2t+1}\}$.

\emph{Case~1: $v_{2t+1}\in D$.}
Removing $\{v_{2t},v_{2t+1},\ell_{2t+1}\}$ and restricting $D$ yields a minimum
dominating set of $J_{t-1}$, and this operation is reversible.
Hence this case contributes $b_{t-1}$ choices.

\emph{Case~2: $\ell_{2t+1}\in D$.}
Then $v_{2t+1}\notin D$, and domination of $v_{2t}$ forces $v_{2t-1}\in D$.
After fixing these vertices, the remaining choices lie in the induced subgraph
isomorphic to $J_{t-2}$, again bijectively.
Thus this case contributes $b_{t-2}$ choices.

Consequently $b_t=b_{t-1}+b_{t-2}$ for $t\ge 2$, with $b_0=1$ and $b_1=3$,
so $b_t=F_{t+3}$.
If $n=2k+1$, then $O_n=J_k$ and $\zeta(O_{2k+1})=F_{k+3}$.
If $n=2k$, the unpended endpoint $v_{2k}$ forces $v_{2k-1}$ in every minimum
dominating set, reducing to $J_{k-1}$ and giving $\zeta(O_{2k})=F_{k+1}$.
\end{proof}

Table~\ref{tab:verification} lists the values of $\gam$ and $\zet$ for $2\le n\le 10$,
confirming the Fibonacci indices in Theorem~\ref{thm:alt} for both even and odd attachment.

\begin{corollary}[Asymptotic behavior of alternating combs]\label{cor:phi}
For alternating combs,
\[
\zeta(E_n)\asymp \varphi^{\,\gamma(E_n)}
\quad\text{and}\quad
\zeta(O_n)\asymp \varphi^{\,\gamma(O_n)},
\]
where $\varphi=(1+\sqrt5)/2$ is the golden ratio.
\end{corollary}

\begin{proof}
Binet's formula gives $F_t=(\varphi^t-\psi^t)/\sqrt5$ with $\psi=(1-\sqrt5)/2$ and $|\psi|<1$; see
\cite[Ch.~6]{GrahamKnuthPatashnik1994}. Hence $F_t=\Theta(\varphi^t)$, and the claimed asymptotics follow
from Theorem~\ref{thm:alt}.
\end{proof}

\section{Complete binary trees and stability under leaf deletions}\label{sec:binary}
The path-based families in Section~2 exhibit three attachment-driven behaviors:
independent choice, boundary-induced rigidity, and Fibonacci coupling.
We now turn to a genuinely hierarchical family.
Complete binary trees display a qualitatively different phenomenon:
a rigid periodic law for $\zet$ that persists despite exponential growth in $|V|$.
We then quantify how this rigidity degrades under controlled boundary perturbations.

\subsection{Period-$3$ dominion in complete binary trees}
Let $T_h$ denote the complete binary tree of height $h\ge 1$, rooted at level $0$, with level sets
$L_i=\{v:\operatorname{dist}(v,\text{root})=i\}$. Then $|L_i|=2^i$ and $|V(T_h)|=2^{h+1}-1$.

\begin{theorem}[Domination number and periodic dominion of complete binary trees]\label{thm:binary}
For $h\ge 1$,
\[
\gamma(T_h)=\left\lfloor\frac{2^{h+2}+3}{7}\right\rfloor,
\qquad
\zeta(T_h)=
\begin{cases}
3, & h\equiv 0 \pmod{3}\ \text{and } h\ge 3,\\
1, & \text{otherwise}.
\end{cases}
\]
\end{theorem}

\begin{proof}
The formula for $\gamma(T_h)$ is classical; see, for example,
Haynes--Hedetniemi--Slater~\cite{HaynesHedetniemiSlater1998}.
We therefore focus on the dominion $\zeta(T_h)$.

We first note that minimum dominating sets may be assumed to avoid leaves.
Indeed, let $h\ge 2$ and suppose a $\gamma$-set $D$ contains a leaf
$\ell\in L_h$ with parent $p\in L_{h-1}$.
Since $N[\ell]\subseteq N[p]$, replacing $\ell$ by $p$ preserves both
cardinality and domination.
Iterating this replacement yields a $\gamma$-set $S$ with
\[
S\cap L_h=\varnothing .
\]

Fix such a set $S$.
If a vertex $x\in L_{h-1}$ has two leaf children $\ell_1,\ell_2\in L_h$,
then domination of both leaves forces $x\in S$, as $S$ contains no leaves
and $N[\ell_i]=\{\ell_i,x\}$.
Consequently,
\[
L_{h-1}\subseteq S
\quad\text{for every $\gamma$-set $S$ with $S\cap L_h=\varnothing$.}
\]

Assume now that $h\ge 4$.
Let $T'$ denote the induced subtree on levels $0$ through $h-3$,
which is isomorphic to $T_{h-3}$.
Since the vertices in $L_{h-1}$ dominate all vertices in levels
$h-2$, $h-1$, and $h$, the remaining domination constraints for $S$
are entirely internal to $T'$.

Consider the intersection $S\cap V(T')$.
Every vertex of $T'$ is dominated either internally or by adjacency
within $T'$, and the size satisfies
\[
|S\cap V(T')|
=|S|-|L_{h-1}|
=\gamma(T_h)-2^{h-1}
=\gamma(T_{h-3}),
\]
where the last equality follows from the closed form for $\gamma$.
Thus $S\cap V(T')$ is a minimum dominating set of $T'$.

Conversely, given any $\gamma$-set $S'$ of $T'$, the set
\[
S'\cup L_{h-1}
\]
dominates $T_h$, since $S'$ dominates all vertices of $T'$
and $L_{h-1}$ dominates every vertex in the top three levels.
Moreover, its cardinality is
\[
|S'|+|L_{h-1}|=\gamma(T_{h-3})+2^{h-1}=\gamma(T_h),
\]
so it is a $\gamma$-set of $T_h$.

These two constructions are inverse to one another and establish a
bijection between minimum dominating sets of $T_h$ and those of $T_{h-3}$.
Therefore,
\[
\zeta(T_h)=\zeta(T_{h-3}) \qquad (h\ge 4).
\]

A direct check yields $\zeta(T_1)=\zeta(T_2)=1$ and $\zeta(T_3)=3$.
The stated period-$3$ behavior follows immediately by induction on $h$.
\end{proof}

The three minimum dominating sets of $T_3$ are illustrated in
Figure~\ref{fig:dominion-behaviors}(d).

\subsection{An envelope for dominion under leaf deletions}
We now quantify how much flexibility is introduced when the boundary structure
of a complete binary tree is perturbed.

For $X\subseteq L_h$, write $T_h-X$ for the induced subtree obtained by deleting
the leaves in $X$.
Let
\[
m_1(X)=\#\{\text{parents in }L_{h-1}\text{ with exactly one deleted child}\}.
\]

\begin{theorem}[Dominion envelope under leaf deletions]\label{thm:envelope}
Let $h\ge 2$ and let $X\subseteq L_h$. Then
\[
\zeta(T_h-X)\le 2^{m_1(X)}\,\zeta(T_h).
\]
\end{theorem}

\begin{proof}
Let $D$ be a minimum dominating set of $T_h-X$.
From $D$, construct a set $\pi(D)\subseteq V(T_h)$ as follows:
whenever $D$ contains a leaf $u\in L_h\setminus X$, replace $u$ by its parent
$p(u)\in L_{h-1}$; all other vertices of $D$ are left unchanged.

Since $N[u]\subseteq N[p(u)]$, this replacement preserves domination,
and hence $\pi(D)$ dominates $T_h$.
Moreover,
\[
|\pi(D)|\le |D|=\gamma(T_h-X)\le \gamma(T_h).
\]
By minimality of $\gamma(T_h)$, it follows that $|\pi(D)|=\gamma(T_h)$,
so $\pi(D)$ is a minimum dominating set of $T_h$.

Fix now a $\gamma$-set $S$ of $T_h$.
Consider a vertex $p\in L_{h-1}$ whose two children are leaves in $T_h$.
If neither child is deleted, or both are deleted, then the domination constraint
at $p$ is forced exactly as in $T_h$.
If, however, exactly one child of $p$ is deleted, then there are at most two
ways for a preimage $D$ of $S$ to realize domination locally at $p$:
either $p\in D$, or (when admissible) $p$ is replaced by its remaining child.

Each such parent contributes at most one independent binary choice, and no
other vertex contributes additional freedom.
Hence, for a fixed $S$, the preimage $\pi^{-1}(S)$ has cardinality at most
$2^{m_1(X)}$.

Summing over all $\zeta(T_h)$ minimum dominating sets of $T_h$ yields
\[
\zeta(T_h-X)\le 2^{m_1(X)}\,\zeta(T_h),
\]
as claimed.
\end{proof}

\begin{corollary}[Sharpness of the envelope bound]\label{cor:single-leaf}
Let $h\ge 2$ and let $\ell\in L_h$ be a leaf of $T_h$. Then
\[
\gamma(T_h-\ell)=\gamma(T_h),
\qquad
\zeta(T_h-\ell)=2\,\zeta(T_h).
\]
\end{corollary}

\begin{proof}
Since exactly one parent in $L_{h-1}$ loses a single child, we have $m_1(\{\ell\})=1$.
Theorem~\ref{thm:envelope} therefore yields the upper bound
\[
\zeta(T_h-\ell)\le 2\,\zeta(T_h).
\]

Let $p$ denote the parent of $\ell$ and let $\ell'$ be its sibling.
In the full tree $T_h$, the pair $\{\ell,\ell'\}$ forces $p$ in every minimum dominating set:
if $p\notin S$, then both $\ell$ and $\ell'$ must belong to $S$, contradicting minimality.
Thus $p$ belongs to every $\gamma$-set of $T_h$.

Fix a minimum dominating set $S$ of $T_h$.
In $T_h-\ell$, the vertex $p$ may either remain in the dominating set, or be replaced by $\ell'$,
and both choices yield minimum dominating sets of equal size.
Hence each $\gamma$-set of $T_h$ gives rise to exactly two distinct minimum dominating sets of $T_h-\ell$,
namely
\[
S
\quad\text{and}\quad
(S\setminus\{p\})\cup\{\ell'\}.
\]
This shows $\zeta(T_h-\ell)\ge 2\,\zeta(T_h)$, and equality follows.

Finally, the constructions above preserve cardinality, so $\gamma(T_h-\ell)\le\gamma(T_h)$.
The reverse inequality is immediate, since adding $\ell$ to any dominating set of $T_h-\ell$
produces a dominating set of $T_h$.
Therefore $\gamma(T_h-\ell)=\gamma(T_h)$.
\end{proof}

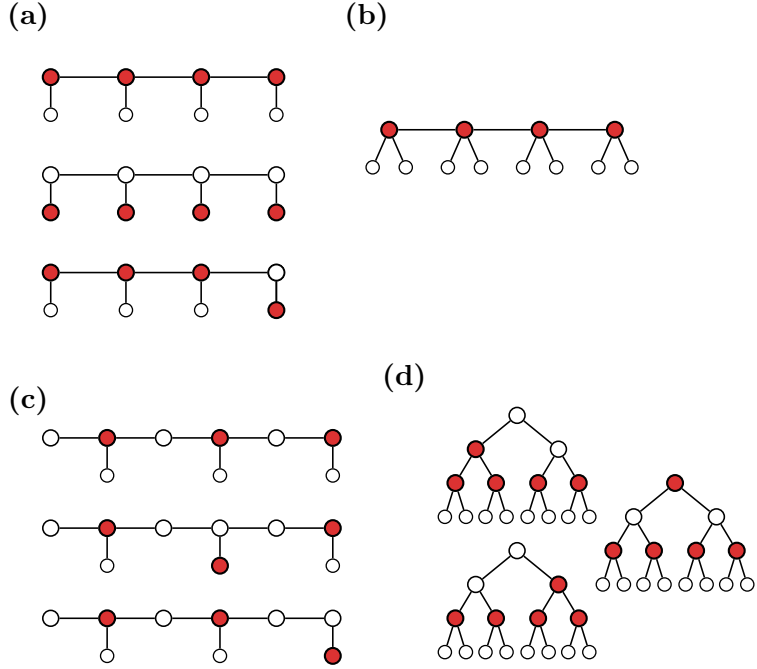
\begin{figure}[t]
\centering
\begin{tikzpicture}[scale=1.0]

\node[font=\bfseries] at (-0.3,4.3) {(a)};

\foreach \x in {0,1,2,3} {
  \node[dom] (v1\x) at (\x*1.0,3.5) {};
  \node[leaf] (l1\x) at (\x*1.0,3) {};
  \draw[edge] (v1\x) -- (l1\x);
}
\foreach \x in {0,1,2} {
  \draw[edge] (v1\x) -- (v1\the\numexpr\x+1);
}

\foreach \x in {0,1,2,3} {
  \node[vertex] (v2\x) at (\x*1.0,2.2) {};
  \node[dom] (l2\x) at (\x*1.0,1.7) {};
  \draw[edge] (v2\x) -- (l2\x);
}
\foreach \x in {0,1,2} {
  \draw[edge] (v2\x) -- (v2\the\numexpr\x+1);
}

\foreach \x in {0,1,2,3} {
  \node[dom] (v3\x) at (\x*1.0,0.9) {};
  \node[leaf] (l3\x) at (\x*1.0,0.4) {};
  \draw[edge] (v3\x) -- (l3\x);
}
\draw[edge] (v30) -- (v31);
\draw[edge] (v31) -- (v32);
\node[vertex] (v33alt) at (3*1.0,0.9) {};
\node[dom] (l33alt) at (3*1.0,0.4) {};
\draw[edge] (v33alt) -- (l33alt);
\draw[edge] (v32) -- (v33alt);

\node[font=\bfseries] at (4.2,4.3) {(b)};

\foreach \x in {0,1,2,3} {
  \node[dom] (v4\x) at ({4.5+\x*1.0},2.8) {};
  \node[leaf] (l4\x-1) at ({4.5+\x*1.0-0.22},2.3) {};
  \node[leaf] (l4\x-2) at ({4.5+\x*1.0+0.22},2.3) {};
  \draw[edge] (v4\x) -- (l4\x-1);
  \draw[edge] (v4\x) -- (l4\x-2);
}
\foreach \x in {0,1,2} {
  \draw[edge] (v4\x) -- (v4\the\numexpr\x+1);
}

\node[font=\bfseries] at (-0.3,-0.8) {(c)};

\foreach \x in {0,1,2,3,4,5} {
  \ifnum\x=1 \node[dom] (v5\x) at (\x*0.75,-1.3) {}; \fi
  \ifnum\x=3 \node[dom] (v5\x) at (\x*0.75,-1.3) {}; \fi
  \ifnum\x=5 \node[dom] (v5\x) at (\x*0.75,-1.3) {}; \fi
  \ifnum\x=0 \node[vertex] (v5\x) at (\x*0.75,-1.3) {}; \fi
  \ifnum\x=2 \node[vertex] (v5\x) at (\x*0.75,-1.3) {}; \fi
  \ifnum\x=4 \node[vertex] (v5\x) at (\x*0.75,-1.3) {}; \fi
}
\foreach \x in {1,3,5} {
  \node[leaf] (l5\x) at (\x*0.75,-1.8) {};
  \draw[edge] (v5\x) -- (l5\x);
}
\foreach \x in {0,1,2,3,4} {
  \draw[edge] (v5\x) -- (v5\the\numexpr\x+1);
}

\foreach \x in {0,1,2,3,4,5} {
  \ifnum\x=1 \node[dom] (v6\x) at (\x*0.75,-2.5) {}; \fi
  \ifnum\x=5 \node[dom] (v6\x) at (\x*0.75,-2.5) {}; \fi
  \ifnum\x=0 \node[vertex] (v6\x) at (\x*0.75,-2.5) {}; \fi
  \ifnum\x=2 \node[vertex] (v6\x) at (\x*0.75,-2.5) {}; \fi
  \ifnum\x=3 \node[vertex] (v6\x) at (\x*0.75,-2.5) {}; \fi
  \ifnum\x=4 \node[vertex] (v6\x) at (\x*0.75,-2.5) {}; \fi
}
\foreach \x in {1,3,5} {
  \ifnum\x=3
    \node[dom] (l6\x) at (\x*0.75,-3.0) {};
  \else
    \node[leaf] (l6\x) at (\x*0.75,-3.0) {};
  \fi
  \draw[edge] (v6\x) -- (l6\x);
}
\foreach \x in {0,1,2,3,4} {
  \draw[edge] (v6\x) -- (v6\the\numexpr\x+1);
}

\foreach \x in {0,1,2,3,4,5} {
  \ifnum\x=1 \node[dom] (v7\x) at (\x*0.75,-3.7) {}; \fi
  \ifnum\x=3 \node[dom] (v7\x) at (\x*0.75,-3.7) {}; \fi
  \ifnum\x=0 \node[vertex] (v7\x) at (\x*0.75,-3.7) {}; \fi
  \ifnum\x=2 \node[vertex] (v7\x) at (\x*0.75,-3.7) {}; \fi
  \ifnum\x=4 \node[vertex] (v7\x) at (\x*0.75,-3.7) {}; \fi
  \ifnum\x=5 \node[vertex] (v7\x) at (\x*0.75,-3.7) {}; \fi
}
\foreach \x in {1,3,5} {
  \ifnum\x=5
    \node[dom] (l7\x) at (\x*0.75,-4.2) {};
  \else
    \node[leaf] (l7\x) at (\x*0.75,-4.2) {};
  \fi
  \draw[edge] (v7\x) -- (l7\x);
}
\foreach \x in {0,1,2,3,4} {
  \draw[edge] (v7\x) -- (v7\the\numexpr\x+1);
}

\node[font=\bfseries] at (4.7,-0.5) {(d)};

\begin{scope}[xshift=6.2cm,yshift=-1.0cm]
  \node[vertex] (r1) at (0,0) {};
  \node[dom] (a1) at (-0.55,-0.45) {};
  \node[vertex] (b1) at (0.55,-0.45) {};
  \node[dom] (c1) at (-0.82,-0.9) {};
  \node[dom] (d1) at (-0.28,-0.9) {};
  \node[dom] (e1) at (0.28,-0.9) {};
  \node[dom] (f1) at (0.82,-0.9) {};
  \node[leaf] (g1) at (-0.96,-1.35) {};
  \node[leaf] (h1) at (-0.68,-1.35) {};
  \node[leaf] (i1) at (-0.42,-1.35) {};
  \node[leaf] (j1) at (-0.14,-1.35) {};
  \node[leaf] (k1) at (0.14,-1.35) {};
  \node[leaf] (l1) at (0.42,-1.35) {};
  \node[leaf] (m1) at (0.68,-1.35) {};
  \node[leaf] (n1) at (0.96,-1.35) {};
  \draw[edge] (r1)--(a1) (r1)--(b1);
  \draw[edge] (a1)--(c1) (a1)--(d1) (b1)--(e1) (b1)--(f1);
  \draw[edge] (c1)--(g1) (c1)--(h1) (d1)--(i1) (d1)--(j1);
  \draw[edge] (e1)--(k1) (e1)--(l1) (f1)--(m1) (f1)--(n1);
\end{scope}

\begin{scope}[xshift=6.2cm,yshift=-2.8cm]
  \node[vertex] (r2) at (0,0) {};
  \node[vertex] (a2) at (-0.55,-0.45) {};
  \node[dom] (b2) at (0.55,-0.45) {};
  \node[dom] (c2) at (-0.82,-0.9) {};
  \node[dom] (d2) at (-0.28,-0.9) {};
  \node[dom] (e2) at (0.28,-0.9) {};
  \node[dom] (f2) at (0.82,-0.9) {};
  \node[leaf] (g2) at (-0.96,-1.35) {};
  \node[leaf] (h2) at (-0.68,-1.35) {};
  \node[leaf] (i2) at (-0.42,-1.35) {};
  \node[leaf] (j2) at (-0.14,-1.35) {};
  \node[leaf] (k2) at (0.14,-1.35) {};
  \node[leaf] (l2) at (0.42,-1.35) {};
  \node[leaf] (m2) at (0.68,-1.35) {};
  \node[leaf] (n2) at (0.96,-1.35) {};
  \draw[edge] (r2)--(a2) (r2)--(b2);
  \draw[edge] (a2)--(c2) (a2)--(d2) (b2)--(e2) (b2)--(f2);
  \draw[edge] (c2)--(g2) (c2)--(h2) (d2)--(i2) (d2)--(j2);
  \draw[edge] (e2)--(k2) (e2)--(l2) (f2)--(m2) (f2)--(n2);
\end{scope}

\begin{scope}[xshift=8.3cm,yshift=-1.9cm]
  \node[dom] (r3) at (0,0) {};
  \node[vertex] (a3) at (-0.55,-0.45) {};
  \node[vertex] (b3) at (0.55,-0.45) {};
  \node[dom] (c3) at (-0.82,-0.9) {};
  \node[dom] (d3) at (-0.28,-0.9) {};
  \node[dom] (e3) at (0.28,-0.9) {};
  \node[dom] (f3) at (0.82,-0.9) {};
  \node[leaf] (g3) at (-0.96,-1.35) {};
  \node[leaf] (h3) at (-0.68,-1.35) {};
  \node[leaf] (i3) at (-0.42,-1.35) {};
  \node[leaf] (j3) at (-0.14,-1.35) {};
  \node[leaf] (k3) at (0.14,-1.35) {};
  \node[leaf] (l3) at (0.42,-1.35) {};
  \node[leaf] (m3) at (0.68,-1.35) {};
  \node[leaf] (n3) at (0.96,-1.35) {};
  \draw[edge] (r3)--(a3) (r3)--(b3);
  \draw[edge] (a3)--(c3) (a3)--(d3) (b3)--(e3) (b3)--(f3);
  \draw[edge] (c3)--(g3) (c3)--(h3) (d3)--(i3) (d3)--(j3);
  \draw[edge] (e3)--(k3) (e3)--(l3) (f3)--(m3) (f3)--(n3);
\end{scope}

\end{tikzpicture}
\caption{Four dominion behaviors in trees. (a)~Full comb $G_4$: $\gamma=4$, 
$\zeta=2^4=16$ (three examples shown). (b)~Double-pendant $G(4,2)$: $\gamma=4$, $\zeta=1$ 
(unique set). (c)~Alternating comb $E_6$: $\gamma=3$, $\zeta=F_4=3$ (all three sets). 
(d)~Binary tree $T_3$: $\gamma=5$, $\zeta=3$ (all three sets). Filled vertices belong to 
dominating sets; hollow vertices are dominated. See Theorems~\ref{thm:uniform}, 
\ref{thm:alt}, and~\ref{thm:binary}.}
\label{fig:dominion-behaviors}
\end{figure}

Since even a single local modification can double the number of optimal solutions without changing
$\gamma$, efficient algorithms are needed to track dominion under perturbations; this motivates the
dynamic programming framework developed in the next section.

\section{Algorithmic aspects}

The structural results established above translate directly into efficient evaluation procedures for
the domination number $\gamma$ and the dominion $\zeta$.
In the path-based families, local forcing leads either to closed-form expressions or to a single
Fibonacci-type recurrence.
For general trees, both parameters can be computed by a standard bottom-up dynamic program with a constant number of local states.

\subsection{Path-based families: closed forms and Fibonacci recurrences}

\begin{theorem}[Evaluation complexity for pendant path families]\label{thm:algo-path}
Let $G$ belong to one of the families $G(n,r)$, $G'_n$, $E_n$, or $O_n$ built from a path on $n$ vertices.
Then $\gamma(G)$ and $\zeta(G)$ can be computed in $O(n)$ time.
Moreover, for $G(n,r)$ and $G'_n$ the values follow in $O(1)$ arithmetic time from $(n,r)$ and $n$ respectively,
and for $E_n$ and $O_n$ the value of $\zeta(G)$ is obtained by evaluating a Fibonacci-type recurrence of
length $\Theta(n)$.
\end{theorem}

\begin{proof}
Theorems~\ref{thm:uniform} and~\ref{thm:modified} provide explicit formulas for $\gamma$ and $\zeta$
on $G(n,r)$ and $G'_n$, so evaluation reduces to reading the input parameters and applying these formulas.
For $E_n$ and $O_n$, Theorem~\ref{thm:alt} expresses $\zeta(G)$ as a Fibonacci number whose index is linear
in $n$.
Evaluating $F_t$ via the recurrence $F_{t+1}=F_t+F_{t-1}$ requires $O(t)=O(n)$ time and constant auxiliary
storage.
\end{proof}

\subsection{Trees under perturbation: linear-time dynamic programming}

\begin{theorem}[Linear-time dynamic programming on trees]\label{thm:tree-dp}
For any tree $T$ on $N$ vertices, both $\gamma(T)$ and the number of minimum dominating sets of $T$
can be computed in $O(N)$ time.
\end{theorem}

\begin{proof}
Root $T$ at an arbitrary vertex $r$.
For each vertex $v$, let $T_v$ be the induced subtree consisting of $v$ and its descendants.
Three local states are maintained, describing the status of $v$ within $T_v$:
$v$ is selected; or $v$ is not selected and is dominated by at least one child; or $v$ is not selected
and is not dominated within $T_v$ (so it must be dominated by its parent outside $T_v$).

For each state $s$ at $v$, store a pair $(m_s(v),c_s(v))$, where $m_s(v)$ is the minimum size of a dominating
set of $T_v$ realizing state $s$, and $c_s(v)$ counts the sets achieving this minimum.
When processing $v$ in postorder, each child subtree contributes independently, and the feasible combinations
of child states are determined by the adjacency constraints needed to realize $s$ at $v$.
The values $m_s(v)$ are computed by minimizing over these combinations, and $c_s(v)$ is obtained by multiplying
child counts and summing over those combinations that attain the minimum.
Since each edge is examined a constant number of times and each vertex performs constant-time state aggregation,
the total running time is $O(N)$.

At the root $r$, only states in which $r$ is dominated are admissible, so $\gamma(T)$ is the minimum of the
corresponding $m_s(r)$ values, and the number of minimum dominating sets is the sum of the corresponding
$c_s(r)$ values over states attaining that minimum.
\end{proof}

\begin{remark}
For the complete binary tree $T_h$, Theorem~\ref{thm:binary} shows that $\zeta(T_h)\in\{1,3\}$ depends
only on $h\bmod 3$, so evaluation is constant time given $h$.
The dynamic program of Theorem~\ref{thm:tree-dp} is intended for perturbed instances such as $T_h-X$,
where this exact periodicity need not hold.
\end{remark}

\begin{table}[ht]
\centering
\caption{Verification of Theorem~\ref{thm:alt} for small cases.}
\label{tab:verification}
\vspace{.1in}
\begin{tabular}{ccccccccc}
\toprule
\multicolumn{4}{c}{$E_n$ (even attachment)} & & \multicolumn{4}{c}{$O_n$ (odd attachment)} \\
\cmidrule{1-4} \cmidrule{6-9}
$n$ & $k$ & $\gam$ & $\zet$ & & $n$ & $k$ & $\gam$ & $\zet$ \\
\midrule
2 & 1 & 1 & $F_2=1$ & & 2 & 1 & 1 & $F_2=1$ \\
3 & 1 & 1 & $F_1=1$ & & 3 & 1 & 2 & $F_4=3$ \\
4 & 2 & 2 & $F_3=2$ & & 4 & 2 & 2 & $F_3=2$ \\
5 & 2 & 2 & $F_2=1$ & & 5 & 2 & 3 & $F_5=5$ \\
6 & 3 & 3 & $F_4=3$ & & 6 & 3 & 3 & $F_4=3$ \\
7 & 3 & 3 & $F_3=2$ & & 7 & 3 & 4 & $F_6=8$ \\
8 & 4 & 4 & $F_5=5$ & & 8 & 4 & 4 & $F_5=5$ \\
9 & 4 & 4 & $F_4=3$ & & 9 & 4 & 5 & $F_7=13$ \\
10 & 5 & 5 & $F_6=8$ & & 10 & 5 & 5 & $F_6=8$ \\
\bottomrule
\end{tabular}
\end{table}

\begin{table}[ht]
\centering
\caption{Dominion behaviors for tree families.}
\label{tab:summary}
\vspace{.1in}
\begin{tabular}{lcc}
\toprule
Structure & $\gam$ & $\zet$ \\
\midrule
Full comb $G_n$ & $n$ & $2^n$ \\
Interior pendants $G'_n$ ($n\ge4$) & $n-2$ & $2^{n-4}$ \\
Even alternating $E_n$ & $\Theta(n)$ & $F_{\Theta(n)}$ \\
Odd alternating $O_n$ & $\Theta(n)$ & $F_{\Theta(n)}$ \\
Multiple pendants ($r\ge2$ per vertex) & $n$ & $1$ \\
Binary tree $T_h$ & $\lfloor(2^{h+2}+3)/7\rfloor$ & period-$3$ in $h$ \\
\bottomrule
\end{tabular}
\end{table}

\section{Conclusion}

We determined exact values of the domination number $\gamma$ and the dominion $\zeta$ for several
families of trees, combining structural characterizations with efficient evaluation procedures.
Across these families, four qualitatively distinct dominion behaviors arise.
Independent local choices at pendant clusters yield exponential growth with $\zeta=2^{\gamma}$.
Coupled local constraints in alternating combs produce Fibonacci growth,
$\zeta\asymp \varphi^{\gamma}$.
Complete binary trees exhibit a rigid periodic behavior, with $\zeta$ depending only on $h \bmod 3$
despite exponential growth in the number of vertices.
Finally, sufficiently high pendant density enforces complete rigidity, collapsing the dominion to
the unique value $\zeta=1$.

These structural regimes admit direct algorithmic interpretation.
Local forcing and bounded interaction reduce the evaluation of $\gamma$ and $\zeta$ to closed forms
or linear-time dynamic programming, enabling efficient computation even for large instances.
This stands in sharp contrast to the general case, where counting minimum dominating sets is
$\#\mathrm{P}$-complete.
The resulting algorithms support applications in network reconfiguration, facility location under
uncertainty, and distributed consensus, where the dominion quantifies structural flexibility and
robustness.

Table~\ref{tab:summary} summarizes the four dominion behaviors identified here and highlights the
transition between exponential, subexponential, periodic, and rigid enumeration.
For the alternating families $E_n$ and $O_n$, the Fibonacci formulas are supported both theoretically
and computationally, with explicit verification provided in Table~\ref{tab:verification}.

Several directions remain open.
Generalizing Fibonacci-type recurrences to broader classes of sparse pendant patterns may reveal
additional families with subexponential dominion growth.
A sharper characterization of the boundary between bounded and unbounded $\zeta$ in nearly complete
trees would clarify how local forcing scales to global rigidity.
From an algorithmic perspective, extending the linear-time framework to wider classes of graphs,
including bounded-treewidth graphs or graphs admitting local decompositions, would broaden its
applicability.
Finally, understanding dominion behavior under edge deletions or more general local modifications,
beyond leaf removal, may further illuminate the stability and fragility of domination-based
structures in networked systems.


\end{document}